\documentclass[11pt]{article}
\usepackage{geometry}                
\usepackage{graphicx}
\usepackage{amssymb}
\usepackage{epstopdf}
\usepackage{amsmath,amsthm,amscd,amssymb}
\usepackage{latexsym}
\usepackage[colorlinks,citecolor=red,pagebackref,hypertexnames=false]{hyperref}

\numberwithin{equation}{section}

\theoremstyle{plain}
\newtheorem{theorem}{Theorem}[section]

\newtheorem{corollary}[theorem]{Corollary}
\newtheorem{proposition}[theorem]{Proposition}

\theoremstyle{definition}
\newtheorem{definition}[theorem]{Definition}

\newtheorem{example}[theorem]{Example}

\theoremstyle{remark}
\newtheorem{remark}[theorem]{Remark}

\newtheorem{case[theorem]}{Case}

\title{A quantitative version of the Steinhaus theorem} 
\author{A. Iosevich\thanks{Supported in part by the National Science Foundation under grant NSF DMS - 2154232 } \ and J. Pakianathan}

\begin{document}
\maketitle

\begin{abstract} The classical Steinhaus theorem (\cite{Steinhaus1920}) says that if $A \subset {\Bbb R}^d$ has positive Lebesgue measure than $A-A=\{x-y: x,y \in A\}$ contains an open ball. We obtain some quantitative lower bounds on the size of this ball and in some cases, relate it to natural geometric properties of $\partial A$. We also study the process $K_n =\frac{1}{2}(K_{n-1} - K_{n-1})$ when $K_0$ is a compact subset of $\mathbb{R}^d$ and determine various aspects of its convergence to $Conv(K_1)$, the convex hull of $K_1$. We discuss some connections with convex geometry, Weyl tube formula and the Kakeya needle problem.

		

\end{abstract}                

\section{Introduction} 

\vskip.125in 

The classical Steinhaus theorem (\cite{Steinhaus1920}) says that if $A \subset {\Bbb R}^d$ has positive Lebesgue measure than $A-A=\{x-y: x,y \in A\}$ contains an open ball. The radius of this ball may be arbitrarily small with a fixed value of the Lebesgue measure of $A$, as can be seen by considering a tube of width $\epsilon^{\frac{1}{d-1}}$ and length $\epsilon^{-1}$ in $\mathbb{R}^d$. 

\begin{definition} Given $A \subset {\Bbb R}^d$, we define the Steinhaus radius of $A$, denoted by $r(A)$ via
$$ r(A)=\sup_{r \geq 0} \{r | A-A \text{ contains a closed ball of radius } r \text{ centered at the origin} \}, $$ where here, and throughout, 
$$ A-A=\{a-a': a,a' \in A\}.$$
 \end{definition} 

\vskip.125in 

The first result of this paper is obtain a simple quantitative bound for the Steinhaus radius of a set of positive Lebesgue measure in terms of the natural geometric properties of the underlying set: 

\vskip.125in 

\begin{theorem} \label{steinhausradiusthm} Let $U \subset {\Bbb R}^d$ be of finite positive Lebesgue measure. Let 
$$M_U(x)= |\{(U-x) \backslash U \cup U \backslash (U-x)\}|$$ and define
$$ M^{*}(U)=\sup_{|x|>0} \frac{M_U(x)}{|x|}.$$ 

Then 
\begin{equation} \label{iso} r(U) \ge \frac{|U|}{M^{*}(U)}. \end{equation} 

\end{theorem} 

\vskip.125in 

In addition to this result, we will study the Banach-Steinhaus Process which consists 
of selecting a compact, nonempty $K_0 \subseteq \mathbb{R}^d$ and iteratively defining  $K_n=\frac{1}{2}(K_{n-1}-K_{n-1})$ for all $n \geq 1$. We show that under the Hausdorff metric on the space of nonempty compact subsets of $\mathbb{R}^d$, this sequence converges to the convex hull of $K_1$, denoted by $Conv(K_1)$, and establish some properties of this convergence. In this context, $r(K_0)$ is the diameter of the largest closed ball about the origin that is contained in $K_1$, and more generally, $r(K_n)$ is the diameter of the largest closed ball about the origin that is contained in $K_{n+1}$.

In particular, when $K_0$ has positive Lebesgue measure, $\lim_{n \to \infty} K_n$ is the closed unit ball for a unique norm on $\mathbb{R}^d$. In this case, in 1-dimension, we also show that $1_{K_n}$ (the indicator function of $K_n$) converges to $1_{Conv(K_1)}$ in $L^1(\mathbb{R}^1)$, so that 
$$\int_{K_n} g dx \to \int_{Conv(K_1)} g dx$$ for any $g \in L^{\infty}(Conv(K_1))$, and, in particular, 
$$\lim_{n \to \infty} Vol(K_n)=Vol(Conv(K_1)).$$ 

Such $L^1$-convergence does not generally hold in the measure zero or higher dimensional cases and quantitative bounds on the rate of convergence based on the diameter $D$ and $r(K_0)$ of $K_0$ are given.

In dimensions $d >1$, if $K_0$ has positive Lebesgue measure, and $Star_n$ denotes the largest star-convex subset with respect to the origin that is contained in $K_n$, then we also show $Star_n \to Conv(K_1)$ in the Hausdorff metric. Thus for every continuous function $f: Conv(K_1) \to \mathbb{R}$ there exists $m \leq M$ such that for any $\epsilon > 0$, we have $$[m+\epsilon, M-\epsilon] \subseteq f(K_n) \subseteq [m, M]$$ for all large enough $n$.

\section{Some initial observations}

\begin{remark} By the Jordan Hypersurface separation theorem (see \cite{GuilleminPollack1974}), a nonempty, smooth compact, connected $(d-1)$ dimensional submanifold $H$ of $\mathbb{R}^d$ has $\mathbb{R}^d-H$ consist of two (path) components, one of which is bounded (the "inside" of $H$) and one of which is unbounded (the "outside" of $H$). Furthermore, the closure of the "inside", is a compact, connected, smooth manifold with boundary $U$, whose topological boundary $\partial U$ is $H$. \\

It is clear in this case, that elements in the symmetric difference $$(U-x)\backslash U \cup U \backslash (U-x)$$ are a subset of $\{y \in \mathbb{R^d} | d(y, \partial U) \leq |x| \}$ as to cross from the "inside" to the "outside" or vice versa you have to cross the separating hypersurface $H=\partial U$. \\

It follows, in the language of Theorem~\ref{steinhausradiusthm}, that 
$$M_U(x) \leq | \{y \in \mathbb{R}^d | d(y, \partial U) \leq |x| \} |.$$ 

\vskip.125in 

Now recall that for any $0 \leq m \leq d$, the $m$-dimensional upper-Minkowski content of $A \subset \mathbb{R}^d$, is defined as 
$$M^{*m}(A)=\limsup_{|x| \to 0^+} \frac{|\{ y | d(y,A) < |x| \}|}{\alpha(d-m)|x|^{d-m}}.$$ Here $\alpha(k)$ is the volume of the $k$-dimensional Euclidean ball of radius $1$. \\

Thus we get that $$\limsup_{|x| \to 0^+} \frac{M_U(x)}{|x|} \leq \limsup_{|x| \to 0^+} \frac{ | \{ y | d(y, \partial U) \leq |x| \} |}{|x|}=2M^{*(d-1)}(\partial U)$$ 
where we have used that $\alpha(1)=2$. (Note also $\{ y \in \mathbb{R}^d | d(y, \partial U)=|x| \}$ will consist of the union of two diffeomorphic copies of $\partial U$ for small $|x|$ and hence have Lebesgue measure zero. This follows from the tubular neighborhood theorem and the fact that in such a neigbourhood of $\partial U$, any vector $y$ can be written as the sum of the closest element $z$ in $\partial U$ to $y$ and an element $w$ in the normal space $N_y(\partial U)$ with $d(y, \partial U)=|w|$. See \cite{GuilleminPollack1974}).

Since $H=\partial U$ is a closed set, and locally the graph image of a $C^1$, hence Lipschitz, function over $\mathbb{R}^{d-1}$, the $(d-1)$ dimensional Minkowski content of $H=\partial U$ agrees with the $(d-1)$-dimensional Hausdorff measure of $\partial U$ (see Theorem 3.2.39 in \cite{FedererHerbert1969}) and 
we may conclude:
$$\limsup_{|x| \to 0^+} \frac{M_U(x)}{|x|} \leq 2{\mathcal H}^{(d-1)}(\partial U) < \infty.$$

Since $M_U(x) \leq 2|U| < \infty$, it follows that for any $\gamma > 0$, if we let $\epsilon(\gamma) > 0$ be such that $0 < |x| < \epsilon(\gamma)$ implies 
$$\frac{M_U(x)}{|x|} < 2{\mathcal H}^{(d-1)}(\partial U) + \gamma,$$ we will then have 
$$M^*(U)=\sup_{|x| > 0} \frac{M_U(x)}{|x|} \leq max(2{\mathcal H}^{(d-1)}(\partial U) + \gamma, \frac{2|U|}{\epsilon(\gamma)})$$ will be finite in this case.

On a more refined level, we have from the Weyl tube formula (see \cite{Weyl1939}, \cite{Gray2004}), the following equality (that holds for all $r$ small enough to guarantee that $\{ y \in \mathbb{R}^d,  d(y, \partial U)=s\}$ is still a smoothly embedded hypersurface for all $0 \leq s \leq r$):
$$
| \{y \in \mathbb{R}^d | d(y, \partial U) \leq r \} | = 2r \sum_{c=0}^{[\frac{d-1}{2}]} \frac{k_{2c}(P)r^{2c}}{(3)(5)(7)\dots(2c+1)}
$$
where $P=\partial U$ with the Riemannian metric induced from the ambient dot product, 
$$k_0(P)=Vol(P), k_2(P)=\frac{1}{2} \int \tau dP$$ where $\tau$ is the scalar curvature of $P$ and in general $k_{2c}(P)$ are integrals of quantities involving more complicated curvature functions of $P$. By general Gauss-Bonnet theorems, when $d=2p+1$ is odd, the highest 
$$k_{2p}(P)=(2 \pi)^p \chi(P)$$ where $\chi(P)$ is the Euler characteristic of $P$ as $P=\partial U$ is compact without boundary and even dimensional in this case (see \cite{Gray2004}). Let $$r_c=\inf \{ s > 0 | \{ y \in \mathbb{R}^d | d(y, \partial U)=s \}  \text{ is not a smoothly embedded submanifold of } \mathbb{R}^d \}.$$ Then $r_c > 0$ by the tubular neighborhood theorem and the Weyl tube formula holds for all $r < r_c$. Often $r_c$ is the smallest distance of which some point in $\mathbb{R}^d$ has more than one (i.e. a nonunique) closest point on $\partial U$ of that distance. Using the Weyl tube lemma gives us the following upper bound for the quantity $M^*(U)$:
$$
M^*(U) \leq \sup_{r > 0} \frac{ |\{ y \in \mathbb{R}^d | d(y, \partial U) \leq r \}}{r} \leq max(\sup_{0 < r < r_c} 2\sum_{c=0}^{[\frac{d-1}{2}]} \frac{k_{2c}(P)r^{2c}}{(3)(5)(7)\dots(2c+1)}, \frac{2|U|}{r_c})
$$
where $[\frac{d-1}{2}]$ is the round-down of $\frac{d-1}{2}$.
This gives a useful quantitative upper bound for $M^*(U)$ which then yields a similar lower bound for the Steinhaus radius $r(U)$ via theorem~\ref{steinhausradiusthm}.
\end{remark}

\begin{example}
Let $\partial U$ be a smooth Jordan curve (smoothly embedded circle) in $\mathbb{R}^2$ and $U$ be the closure of its inside. Then $k_0(\partial U)$ is the arclength of $\partial U$ and hence the perimeter length of $U$ and the Weyl tube lemma mentioned in the previous remark yields:
$$
M^*(U) \leq max(2 Perimeter(U), \frac{2Area(U)}{r_c})
$$

\end{example}

\begin{example}
Let $\partial U$ be a closed surface of genus $g$ smoothly embedded in $\mathbb{R}^3$ and let $U$ be the closure of its inside. Then $k_0(\partial U)$ is the surface area of the boundary of $U$,$k_2(\partial U) = \frac{1}{2} \int K dVol$ where $K$ is the Gaussian curvature of $\partial U$. Thus $k_0(\partial U)$ equals $ 2\pi \chi(\partial U) = 2\pi (2-2g)$ by the Gauss-Bonnet theorem. This yields:
$$
M^*(U) \leq max(\sup_{ 0 < r < r_c} (2 (SurfaceArea(\partial U) +\frac{ 2\pi(2-2g)r^2}{3}), \frac{2Vol(U)}{r_c})
$$
and so for $g \geq 1$ (non-spheres) we have:
$$
M^*(U) \leq max(2 SurfaceArea(\partial U), \frac{2Vol(U)}{r_c})
$$
while for (arbitrary smooth embeddings of) spheres ($g=0$) we have:
$$
M^*(U) \leq max(2 (SurfaceArea(\partial U) + \frac{4\pi r_c^2}{3}), \frac{2Vol(U)}{r_c})
$$
\end{example}

To move beyond domains $U$ with smooth boundaries, we can also consider convex bodies, i.e. compact, convex subsets of $\mathbb{R}^d$ with nonempty interior. It is a basic fact that such bodies are homeomorphic to the closed unit ball and so their boundaries are homeomorphic to the $d-1$-dimensional sphere $S^{d-1}$ but they need not be smoothly embedded submanifolds as is the case for example for $n$-gons embedded in the plane.

In this case of convex bodies (with boundary that is not necessarily smooth), there is a formula similar to the Weyl tube formula for the volume of a tube about $\partial U$ which is called Steiner's formula (see Chapter 10 of \cite{Gray2004}).
Some cases of Steiner's formula are: 

If $B$ is convex body in the plane and $B_r$ is the set of points of distance $\leq r$ to $B$ then for all $r \geq 0$:
$$
Area(B_r)=Area(B) + length(\partial B) r + \pi r^2,
$$
$$
length(\partial B_r) = length(\partial B) + 2\pi r.
$$
and if $B$ is a convex body in $\mathbb{R}^3$ we have: 
$$
Volume(B_r) = Volume(B) + Area(\partial B) r  + \frac{1}{2} k_1(\partial B)r^2 + \frac{4 \pi r^3}{3},
$$
$$
Area(\partial B_r) = Area(\partial B) + k_1(\partial B) r + 4\pi r^2.
$$
However note that $B_r-B$ is only the "outer" half-tube about $\partial B$ and not the full tube about $\partial B$. In fact $B_r$ is the Minkowski sum of $B$ with a closed ball of radius $r$ about the origin and so is still convex so that these Steiner formulas hold for all positive $r$ unlike the more general Weyl tube formulas. However similar formulas for the volume of the "inner" half-tube do breakdown at some $r=r_c$ as before.

\begin{remark} (Sharp examples) Let $U=[0,\ell_1] \times [0, \ell_2]$, $\ell_1 \geq \ell_2 >0$, be a rectangle in the plane. For $x=(x_1,x_2)$ with $|x_1| \leq \ell_1, |x_2| \leq \ell_2$, it is a direct simple computation to see that 
$$M_U(x)=|\{(U-x) \backslash U \cup U \backslash (U-x)\}|=2(\ell_1 |x_2| + \ell_2 |x_1|-|x_1x_2|).$$ 

\vskip.125in 

In general, $M_U(x)=M_U(y)$ where $y=(y_1,y_2)$ has $y_i = \min(\ell_i,|x_i|), i=1,2$. 

\vskip.125in 

It then follows that 
$$M^*(U)=\sup_{|x| > 0} \frac{M_U(x)}{|x|}=\sup_{x \in U-\{0\}} \frac{2(\ell_1 |x_2| + \ell_2 |x_1|-|x_1x_2|)}{|x|}.$$

\vskip.125in 

When restricted, to the line of slope $0 \leq m < \infty$ in $U$, the quantity in the supremum becomes $\frac{2(\ell_1 m + \ell_2 - m|x_1|)}{\sqrt{1+m^2}}$ for $x_1 \neq 0$ whose supremum on this line segment is then $\frac{2(\ell_1,\ell_2)\cdot (m,1)}{||(m,1)||}$. This is then maximized as we vary $m$, when $(m,1)$ is in the same direction as $(\ell_1,\ell_2)$ giving $M^*(U)=2(\sqrt{\ell_1^2 + \ell_2^2})$. 
Theorem~\ref{steinhausradiusthm} then says $r(U) \geq \frac{\ell_1 \ell_2}{2\sqrt{\ell_1^2+\ell_2^2}}$ which is optimal (up to a factor of 2) when $\ell_1 >> \ell_2$ as $U-U=[-\ell_1,\ell_1] \times [-\ell_2,\ell_2]$.

Also note in this example, that 
$$\limsup_{|x| \to 0} \frac{M_U(x)}{|x|}=\sup_{|x| > 0} \frac{M_U(x)}{|x|}=2\sqrt{\ell_1^2 + \ell_2^2}$$ coincide.

\end{remark} 




Due to the relationship of $\limsup_{|x| \to 0^+} \frac{M_U(x)}{|x|}$ and the $(d-1)$-dimensional Hausdorff measure of $\partial U$ illustrated in the previous discussions, it is clear that if we perturb the closed unit ball by a small fractal perturbation of its boundary, we can get a set $U$ with very large, or even infinite $M^*(U)$ due to the poorly behaved boundary set even though $r(U)$ would not have been significantly changed. This indicates our bound is far from sharp when $\partial U$ deviates from $d-1$-dimensional behavior significantly and motivates some modifications to theorem~\ref{steinhausradiusthm} which we discuss now.

The point is that if $U=U_1 \cup U_2$, then it is obvious that $r(U) \geq r(U_i)$, $i=1,2$. With this in mind, we offer the following simple modification of Theorem \ref{steinhausradiusthm} that it is more stable against badly behaved boundary issues. 

\begin{theorem} \label{mainmodified} Let $U \subset {\Bbb R}^d$ of finite positive Lebesgue measure. Given $S \subset U$, let 
$$M_S(x)= |\{(S-x) \backslash S \cup S \backslash (S-x)\}|$$ and define
$$ M^{*}(S)=\sup_{|x|>0} \frac{M_S(x)}{|x|}.$$ 

Then 
\begin{equation} \label{iso} r(U) \ge \sup_{S \subset U} \frac{|S|}{M^{*}(S)}. \end{equation} 

\end{theorem} 

\vskip.25in 

\section{The Banach-Steinhaus Process}

Let 
$$Cpt_d=\{ A \subseteq \mathbb{R}^d | A \text{ is compact and nonempty }\}.$$

\vskip.125in 

Recall for nonempty $A \subseteq \mathbb{R}^d$ one can define the distance from $x$ to $A$ via 
$$d_A(x)=\inf_{a \in A} \{ ||x-a|| \}.$$ This is a continuous function of $x$ and $d_A(x)=0$ if and only if $x \in \bar{A},$ the closure of $A$.

The Hausdorff distance between $A, B \in Cpt_d$, denoted $d_{\frak{H}}(A,B)$ is defined as 
$$d_{\frak{H}}(A,B)=max(\sup_{b \in B} d_A(b), \sup_{a \in A} d_B(a)).$$

Equivalently, if we denote by $A^{\epsilon} = \{ x \in \mathbb{R}^d | d_A(x) < \epsilon \}$, then $$d_{\frak{H}}(A,B)=\inf \{ \epsilon > 0 | A \subseteq B^{\epsilon} \text{ and } B \subseteq A^{\epsilon} \}.$$

It is well-known that $(Cpt_d, d_{\frak{H}})$ is a complete metric space.

\vskip.125in 

We consider the following process, which we will call the Banach-Steinhaus process on $Cpt_d$. First define the Steinhaus map $S$ via $S(K)=\frac{1}{2}(K-K)$.

We say that $K$ is {\bf symmetric around the origin}, if $K=-K$ and $0 \in K$ and denote 
$$SCpt_d=\{ A \in Cpt_d | A=-A, 0 \in A \}$$ to be the space of compact symmetric sets (under the Hausdorff metric).

We define a Banach-Steinhauss process (next proposition checks well-definedness) as follows. Choose 
$K_0 \in Cpt_d$ and define 
$$K_n=S(K_{n-1})=\frac{1}{2}(K_{n-1}-K_{n-1})$$ inductively for $n \geq 1$. We will show this process converges to the convex hull of $K_1$, i.e., 
$$\lim_{n \to \infty} K_n = Conv(K_1).$$

We collect some of the basic properties of the Steinhaus map and process in the next proposition:

\begin{proposition}
\label{pro:Steinhausmap} For $K \in Cpt_d$, define $S(K)=\frac{1}{2}(K-K)$. Then the following hold. 

\begin{enumerate}
\item For all $K \in Cpt_d$ we have $S(K) \in SCpt_d$.
\item The diameter of $S(K)$ is equal to the diameter of $K$. 
\item If $K \in SCpt_d$, then $S(K)=\frac{1}{2}(K+K) \subseteq Conv(K)$, the convex hull of $K$.
\item Choose $K_0 \in Cpt_d$ and define $K_n=S(K_{n-1})$ inductively for $n \geq 1$ (We refer to this as the Banach-Steinhaus process). Then 
$K_1 \subseteq K_2 \subseteq \dots \subseteq K_n \subseteq Conv(K_1)$ and hence 
$Conv(K_n)=Conv(K_1)$ for all $n \geq 1$. 
\item If $K_n$ is defined as in the last item, then for all $n \geq 1$, 
$$K_n = \{ \sum_{i=1}^N \frac{\alpha_i}{2^{n-1}} x_i | 0 \leq \alpha_i  \text{ integers },\sum_{i=1}^N \alpha_i=2^{n-1}, x_i \in K_1 \}.$$
\item In the Banach-Steinhaus process, $\lim_{n \to \infty} K_n = Conv(K_1)$. In fact, for $$n \geq log_2(d(d+1)),$$ we have $$d_{\frak{H}}(K_n, Conv(K_1)) \leq \frac{Dd}{2^n}$$ where $D$ is the diameter of $K_0$ and $d$ is the ambient Euclidean dimension.
\item A compact set $A \in Cpt_d$ has $S(A)=A$ if and only if $A$ is symmetric and convex.
\item For any $A,B \in Cpt_d$, $d_{\frak{H}}(S(A),S(B)) \leq d_{\frak{H}}(A,B)$. Thus the map $\theta(A)=Conv(S(A))$ is a continuous retraction from $Cpt_d$ to $CSCpt_d$, the space of convex and symmetric compact subsets of $\mathbb{R}^d$.
\end{enumerate}
\end{proposition}

\begin{proof}
Item 1: As $K$ is compact, so is $K \times K$. As $K-K$ is the continuous image of $K \times K$ under the continuous map $-: \mathbb{R}^d \times \mathbb{R}^d \to \mathbb{R}^d$, it follows that $K-K$ is compact. It is easy to see that $K-K$ is also symmetric around the origin.

Item 2: Suppose the diameter of $K$ is $D=\sup \{||k_1-k_2|| | k_1, k_2 \in K \}$.
It follows by definition that $K-K$ is contained in the closed Euclidean ball of radius $D$ about the origin, thus $S(K)=\frac{1}{2}(K-K) \subseteq \bar{B}(0; \frac{D}{2})$.
Thus $Diam(S(K)) \leq D=Diam(\bar{B}_d(0,\frac{D}{2}))$.

As $K$ is compact, there exists $k_1, k_2 \in K$ such that $||k_1-k_2||=D$. 
Then $u=\frac{1}{2}(k_1-k_2)$ and $-u=\frac{1}{2}(k_2-k_1)$ lie in $S(K)$ and have distance $D$ apart so $Diam(S(K)) \geq D$ and so we conclude $Diam(S(K))=D=Diam(K)$.

Item 3: The convex hull $Conv(K)$ in general consists of all convex combinations of elements of $K$ i.e. $Conv(K)=\{ \sum_{i=1}^N t_i x_i | 0 \leq t_i, \sum_{i=1}^N t_i=1, x_i \in K \}$. In particular $Conv(K)$ contains $\frac{1}{2}(K+K)=\{ \frac{1}{2}k_1 + \frac{1}{2}k_2 | k_1, k_2 \in K\}$. When $K$ is symmetric, $K=-K$ so $S(K)=\frac{1}{2}(K-K) = \frac{1}{2}(K+K) \subseteq Conv(K)$.

Item 4: Thus if $K$ is symmetric, $K \subseteq S(K) \subseteq Conv(K)$ forces $Conv(S(K))=Conv(K)$. In a Banach-Steinhaus process $K_0, K_1, K_2, \dots$, as $K_1$ is symmetric, this forces $Conv(K_n)=Conv(K_1)$ for all $n \geq 1$.

Item 5: We prove 
$$K_n = \left\{ \sum_{i=1}^N \frac{\alpha_i}{2^{n-1}} x_i | 0 \leq \alpha_i \text{ integers },\sum_{i=1}^N \alpha_i=2^{n-1}, x_i \in K_1 \right\}$$
by induction on $n$. The case $n=1$ is trivial so assume it has been proven true for some $n \geq 1$ and we will show the inductive step to obtain truth for the $n+1$ case.

As $K_n$ is symmetric as $n \geq 1$, we have $K_{n+1}=S(K_n)=\frac{1}{2}(K_n + K_n)$ and so using the induction hypothesis, the typical element of $K_{n+1}$ will be of the form 
$$\frac{1}{2} \left(\sum_{i=1}^N \frac{\alpha_i}{2^{n-1}}x_i + \sum_{i=1}^N \frac{\beta_i}{2^{n-1}}x_i' \right)$$ for integers $0 \leq \alpha_i, \beta_i$ with
$\sum_{i=1}^N \alpha_i = 2^{n-1} = \sum_{i=1}^N \beta_i$ and $x_i, x_i' \in K_1$. A trivial rearrangement shows this is indeed the desired form for $K_{n+1}$ and finishes the induction.

Item 6: Let $D=Diam(K_0)$. We will obtain an upper bound on $d_{\frak{H}}(K_n, Conv(K_1))$ when $n \geq 1$ in terms of $D, n$ and the ambient dimension $d$.

Let $y \in Conv(K_1)$. As we are in $\mathbb{R}^d$, Caratheodory's theorem, says that $y$ can be expressed as a convex combination of at most $d+1$ elements of $K_1$. 
Thus $y=\sum_{i=0}^d t_i x_i$ with $0 \leq t_i$ real numbers and $\sum_{i=0}^d t_i=1$, $x_i \in K_1$. Without loss of generality, let $t_0$ be the largest of the $t_i$ and so $t_0 \geq \frac{1}{d+1}$.

By item 5, the elements of $K_n$ are similar convex combinations but with the $t_i$ constrained to be nonnegative rational numbers whose denominator is $2^{n-1}$. These rationals mesh the unit interval $[0,1]$ into segments of equal length $\frac{1}{2^{n-1}}$ and so we may select for $1 \leq i \leq d$ such a $q_i$ with $|q_i - t_i| \leq \frac{1}{2^n}$ by taking the nearest mesh point $q_i$ to $t_i$. For the final $q_0$ we are forced to use $q_0=1-\sum_{i=1}^d q_i$.

Note that since 
$$\sum_{i=1}^d t_i = 1 - t_0 \leq 1-\frac{1}{d+1},$$ we have 
$$\sum_{i=1}^d q_i \leq \sum_{i=1}^d t_i + \frac{d}{2^n}$$ will be $\leq 1$ so that $q_0 \geq 0$ as long as $\frac{d}{2^n} \leq \frac{1}{d+1}$ i.e. $n \geq log_2(d(d+1))$.

With this assumption, then we can find a point $x=\sum_{i=0}^d q_i x_i \in K_n$ such that 
$$d(x,y) \leq \sum_{i=0}^d |q_i-t_i| |x_i| \leq \frac{D}{2} \sum_{i=0}^d |q_i-t_i|$$ as 
$K_1 \subseteq \bar{B}_d(0,\frac{D}{2})$.

\vskip.125in 

Thus 
$$d(x,y) \leq \frac{D}{2} \left(|t_0 - q_0| + d \frac{1}{2^n}\right) \leq \frac{D}{2}(2d \frac{1}{2^n})= \frac{Dd}{2^n}.$$ It follows that for any $y \in Conv(K_1)$, $d_{K_n}(y) \leq \frac{Dd}{2^n}$. As $K_n \subseteq Conv(K_1)$, this means $d_{\frak{H}}(K_n, Conv(K_1)) \leq \frac{Dd}{2^n}$ as long as $n \geq log_2(d(d+1))$. This establishes 
$$\lim_{n \to \infty} K_n = Conv(K_1).$$

Item 7: If $K_1$ is compact, symmetric and convex then $K_1 \subseteq S(K_1) \subseteq Conv(K_1)=K_1$ forces $S(K_1)=K_1$ as desired. Conversely if $S(K)=K$ it follows that $K$ is symmetric and that $K_n=S(K_{n-1})=\dots=K_1$ for all $n$ and so by item 6, $K_1=\lim_{n \to \infty} K_n = Conv(K_1)$ is convex.

Item 8: Suppose $\epsilon > 0$ is such that $B \subseteq A^{\epsilon}$ and $A \subseteq B^{\epsilon}$. It is then direct to show for any $a_1, a_2 \in A$, there exists $b_1, b_2$ in $B$ such that $d(b_i, a_i) < \epsilon$ and so $d(\frac{a_1-a_2}{2},\frac{b_1-b_2}{2}) < \epsilon$ which shows that $d_{\frak{H}}(S(A),S(B)) \leq d_{\frak{H}}(A,B)$ and establishes the continuity of the Banach-Steinhaus map $S: Cpt_d \to Cpt_d$. As we have seen from the previous items, that 
$\theta(A)=Conv(S(A))=lim_{n \to \infty} S^{[n]}(A)$ where $S^{[n]}$ is the n-fold composite of $S$ with itself, it follows that $d_{\frak{H}}(\theta(A),\theta(B)) \leq d_{\frak{H}}(A,B)$ and so $\theta$ is also continuous and maps $Cpt_d$ to $CSCpt_d$ and is the identity on the latter by Item 7. This is the definition of a topological retraction.

\end{proof}

\subsection{Banach Steinhaus process in dimension 1}

\begin{example}[2 points in $\mathbb{R}^1$]
\label{example:2points}
Let $K_0$ be a two point set of diameter $D$ in $\mathbb{R}^1$ and consider the Banach Steinhaus Process $K_n=S^{[n]}(K_0)$. Then it is easy to compute that

$$K_1 = \left\{ -\frac{D}{2}, 0, \frac{D}{2} \right\} = \frac{D}{2} \{-1,0,1\},$$
$$K_2 = \frac{D}{2}\{ -1,-\frac{1}{2},0,\frac{1}{2},1\},$$
and, more generally,
$$K_n = \frac{D}{2} \left\{ \text{Set of rational numbers in } [-1,1] \text{ with denominator } 2^{n-1} \right\}.$$

Since $Conv(K_1)=[\frac{-D}{2}, \frac{D}{2}]$ we have $d_{\frak{H}}(K_n,Conv(K_1))=\frac{D}{2^{n+1}}$ showing that the bound in item 5 of the proposition~\ref{pro:Steinhausmap} is sharp up to a factor of $2$.
\end{example}

\begin{remark}
Notice that if $K=\cup_{i=1}^n A_i \subseteq \mathbb{R}^d$ then 
$$ S(K)=\bigcup_{i=1}^n S(A_i) \cup \bigcup_{1 \leq i \neq j \leq n} \frac{1}{2}(A_i-A_j).$$
\end{remark}

\begin{example}[Important Example]
\label{example:important}
Let $r, D, \delta$ be nonnegative real numbers such that $\delta \leq r < D, r > 0$.
(It is important we allow $\delta=0$ for later application).
Then 
$$K_1(r,D,\delta)=\left[-\frac{D}{2}, -\frac{D}{2}+\delta \right] \cup \left[-\frac{r}{2}, \frac{r}{2} \right] \cup \left[\frac{D}{2}-\delta,\frac{D}{2}\right]$$ is a compact symmetric subset of $\mathbb{R}^1$, with diameter $D$, consisting of the union of $3$ closed intervals $A_1,A_2$ and $A_3$ where the middle interval has diameter $r$ and the two outer intervals have diameter $\delta$. Note $Conv(K_1(r,D,\delta))=[-\frac{D}{2},\frac{D}{2}]$ and that these intervals are only disjoint if $r +2\delta < D$. Indeed when $r + 2\delta \geq D, K_1(r,D,\delta_1)=[-\frac{D}{2},\frac{D}{2}]$.

Using the remark before this example and the observation that if $A$ is a closed interval of length $L$ then $S(A)=[-\frac{L}{2},\frac{L}{2}]$ we compute: 
$$K_2(r,D,\delta)=S(K_1(r,D,\delta))=[-\frac{r}{2},\frac{r}{2}] \cup I_1 \cup -I_1 \cup I_2 \cup -I_2$$
where we have used that $\delta \leq r$ and $I_1=[\frac{D}{2}-\delta,\frac{D}{2}]$,
$I_2=[\frac{D-r}{4}-\frac{\delta}{2},\frac{D+r}{4}]$. The length of $I_1$ and $I_2$ are $\delta$ and $\delta_1=\frac{\delta+r}{2} > 0$ (as $r>0$) respectively. Thus
$$
K_2(r,D,\delta) = K_1(r,D,\delta) \cup K_1(r, \frac{D+r}{2}, \delta_1) 
\supseteq K_1(r,D,\delta) \cup K_1(r, \frac{D+r}{2}, \delta).
$$
Let us focus on the containment: 
$$
K_2(r,D,\delta) \supseteq K_1(r, \frac{D+r}{2}, \delta_1).
$$
Let $K_n(r,D,\delta)=S^{[n-1]}(K_1(r,D,\delta))$ as usual. Using the fact that $A \subseteq B$ implies $S(A) \subseteq S(B)$ and $\delta_1 \geq \delta$, we iterate to see that
$$
K_3(r,D,\delta) \supseteq K_1(r, \frac{D+3r}{4}, \delta_1)
$$
and more generally that
$$
K_n(r,D,\delta) \supseteq K_1(r, \frac{D+(2^{n-1}-1)r}{2^{n-1}}, \delta_1)
$$
Thus once $n$ is big enough so that $r+2\delta_1 \geq \frac{D+(2^{n-1}-1)r}{2^{n-1}}$ we have that $K_n(r,D,\delta) \supseteq [-\frac{D+(2^{n-1}-1)r}{2^{n-1}},\frac{D+(2^{n-1}-1)r}{2^{n-1}}]$. Any $n \geq log_2(\frac{D-r}{\delta_1})$ satisfies this so in particular any $n \geq log_2(\frac{2D}{r})$. Let $n_0$ be the roundup of $log_2(\frac{D}{r})$ then $n=n_0+1$ works. As $K_{1+n_0}(r,D,\delta)$ also contains $K_1(r,D,\delta)$, it follows then that 
$$
K_{1+n_0}(r,D,\delta) \supseteq K_1(t_0r+(1-t_0)D, D, \delta).
$$
where $\frac{1}{2} \leq t_0=1-\frac{1}{2^{n_0}}  < 1$. (Note $n_0 \geq 1$ as $D > r$.)
Iterating, one concludes:
$$
K_{1+2n_0}(r,D,\delta) \supseteq K_{1+n_0}(t_0r+(1-t_0)D, D, \delta) \supseteq K_1(t_0^2r+(1-t_0^2)D, D, \delta)
$$
and generally for $\ell \geq 1$, we have:
$$
K_{1+\ell n_0}(r,D,\delta) \supseteq K_1(t_0^{\ell}r+(1-t_0^{\ell})D, D, \delta)
$$

Fix $0 < \epsilon < \frac{D}{2}$. 
When $\ell$ is large enough so that $t_0^{\ell}r+(1-t_0^{\ell})D \geq D-2\epsilon$, it follows that $$[-\frac{D}{2}+\epsilon,\frac{D}{2}-\epsilon] \subseteq K_{1+\ell n_0}(r,D,\delta) \subseteq [-\frac{D}{2},\frac{D}{2}].$$ It suffices to choose $\ell$ large enough so that $t_0^{\ell} \leq \frac{2\epsilon}{D-r}$ which is the case when $\ell \geq \frac{log_2(\frac{D-r}{2\epsilon})}{|log_2(t_0)|}$. Let $\ell_0$ be the round up of $\frac{log_2(\frac{D}{2\epsilon})}{|log_2(t_0)|}$.

Then for all $n > \ell_0 n_0$ we have $[-D/2+\epsilon, D/2-\epsilon] \subseteq K_n(r,D,\delta) \subseteq [-D/2,D/2]$.

\vskip.125in 

In particular, 
$$\lim_{n \to \infty} Vol(K_n(r,D,\delta))=D.$$

\end{example}

\begin{theorem}[$1$ Dimensional Banach-Steinhaus process]
\label{theorem:1DBSProcess}
Let $K_0$ be a compact subset of $\mathbb{R}^1$ with diameter $D$, positive Lebesgue measure and $r=r(K)$ as defined in Theorem~\ref{steinhausradiusthm}. Let $K_n=S^{[n]}(K_0)$ denote the $n$th stage of the Banach-Steinhaus process.

If $r=D$, then $K_n=[-\frac{D}{2},\frac{D}{2}]$ for all $n \geq 1$ so assume $r < D$.

Fix $0 < \epsilon < \frac{D}{2}$. 
Let $n_0$ be the round up of $log_2(D/r)$, $t_0=1-\frac{1}{2^{n_0}}$ and let $\ell_0$ be the round up of $\frac{log_2(\frac{D}{2\epsilon})}{|log_2(t_0)|}$. 

For any $n > \ell_0 n_0$ we have:

$$
[-D/2+\epsilon, D/2-\epsilon] \subseteq K_n \subseteq [-D/2,D/2]
$$
and $||1_{K_n} - 1_{[-\frac{D}{2},\frac{D}{2}]}|| \leq 2 \epsilon$.

In particular $\lim_{n \to \infty} Vol(K_n)=D$ and $\lim_{n \to \infty} 1_{K_n}=1_{[-\frac{D}{2},\frac{D}{2}]}$ in $L^1(\mathbb{R}^1)$.

Then if $K_0^{[2^{n-1}]}=K_0+\dots+K_0$ denotes the $2^{n-1}$-fold sumset of $K_0$ with itself, we have
$$
[-2^{n-1}(D-2\epsilon), 2^{n-1}(D-2\epsilon)] \subseteq K_0^{[2^{n-1}]}-K_0^{[2^{n-1}]} \subseteq [-2^{n-1}D, 2^{n-1}D].
$$
\end{theorem}
\begin{proof}
By assumption, $S(K_0)$ contains $K_1(r,D,0)$ as defined in Example~\ref{example:important} and $r > 0$ due to the classical Steinhaus result. Thus all but the last statement follow by the results in that example.

Observe that 
$$S(K_0)=\frac{K_0-K_0}{2}, \ S^{[2]}(K_0)=S(S(K_0)) = \frac{1}{4}(K_0+K_0-K_0-K_0),$$ and, by induction, 
$$K_n=S^{[n]}(K_0)=\frac{1}{2^{n}}(K_0^{[2^{n-1}]}-K_0^{[2^{n-1}]}).$$

\end{proof}

As Example~\ref{example:2points} shows, the equation $\lim_{n \to \infty} Vol(K_n)=D$ does not hold for compact sets $K_0$ of measure $0$ and diameter $D$ even though $K_n$ still converges to $[-\frac{D}{2},\frac{D}{2}]$ relative to the Hausdorff metric. In this example $1_{K_n}$ do converge to $1_{[\frac{-D}{2},\frac{D}{2}]}$ in $L^1(\mathbb{R}^1)$ so $L^1$-convergence does require positive measure though Hausdorff-metric converge does not.

\subsection{Banach-Steinhaus process in higher dimensions}

To extend Theorem~\ref{theorem:1DBSProcess} to higher dimensions, we will now make some definitions common to the theory of convex sets:

\begin{definition}[Quasi-support functions]
\label{defn:sptfunction}
Let $K \in SCpt_d$. 

For every $\theta \in S^{d-1}$, let $$P_{\theta}=\{ t \in \mathbb{R}_{\geq 0} | t \theta \in K \}$$ be the profile of $K$ in the direction $\theta$. 

Let $r(\theta), D(\theta)$ be the largest (respectively smallest) nonnegative real numbers such that 
$$ \left[0,\frac{r(\theta)}{2}\right] \subseteq P_{\theta} \subseteq \left[0, \frac{D(\theta)}{2}\right].$$

Note, by symmetry $P_{\theta}=P_{-\theta}$ and hence $r(\theta)=r(-\theta), D(\theta)=D(-\theta)$. By compactness of $K$, $0 \leq r(\theta) \leq D(\theta) < \infty$ and $\pm \frac{r(\theta)}{2}\theta, \pm \frac{D(\theta)}{2}\theta \in K$ for all $\theta$. Let $r=\inf_{\theta \in S^{d-1}} r(\theta)$ and $D=\sup_{\theta \in S^{d-1}} D(\theta)$. It is clear that $r$ is the diameter of the largest closed Euclidean ball about the origin contained in $K$ and $D$ is the diameter of the smallest closed Euclidean ball about the origin which contains $K$ and that $D$ coincides with the diameter of $K$. 

We say $K$ is star-convex with respect to the origin, if the line segment between any point of $K$ and the origin, is contained in $K$. This happens if and only if $r(\theta)=D(\theta)$ for all $\theta \in S^{d-1}$.
\end{definition}

Let $Star(K)$ denote the largest star-convex (with respect to origin) subset of $K$, then it is easy to see 
$$Star(K)=\bigcup_{\theta \in S^{d-1}} \left\{ t \theta | 0 \leq t \leq \frac{r(\theta)}{2} \right\}.$$

 Note $0 \in Star(K)$ and $Star(K)$ is always path connected (any two points have a path between them in $Star(K)$ obtained by concatenating straight line segments to the origin). $Star(K)$ is closed in $K$ and hence compact. Indeed suppose $t_n \theta_n \to t \theta \in K$ where $t_n, t \geq 0, \theta_n, \theta \in  S^{d-1}$ and $t_n \theta_n \in Star(K)$ for all $n \geq 1$. It follows for any $0 \leq s \leq 1, st_n \theta_n \in Star(K)$ and so $s t \theta$ is also in the closure of $Star(K)$ in $K$. Thus the closure of $Star(K)$ in $K$ is still star-convex with respect to the origin and hence must coincide with $Star(K)$.

\begin{theorem}
Let $K_0 \in Cpt_d$ of diameter $D$ and let $K_n=S^{[n]}(K_0)$ denote the Banach-Steinhaus process. Assume $K_0$ has {\bf positive Lebesgue measure}. For $n \geq 1$, let $Star_n=Star(K_n)$ be the largest star-convex subset with respect to the origin contained in $K_n$.Then 
$$\lim_{n \to \infty} Star_n = Conv(K_1)$$ with respect to the Hausdorff metric. 

In addition, if $r$ denotes the diameter of the largest closed ball about the origin contained in $K_1$, then $r > 0$ and for $n \geq max(log_2(\frac{2Dd}{r}), log_2(d(d+1)))$ we have $$\frac{1}{2} Conv(K_1) \subseteq K_{n+1} \subseteq Conv(K_1).$$

\end{theorem}
\begin{proof}
Recall $K_1 \subseteq K_2 \subseteq \dots \subseteq Conv(K_1)$. So $Star_1 \subseteq Star_2 \subseteq Star_3 \subseteq \dots \subseteq Conv(K_1)$. 
For $n \geq 1$, let $r_n(\theta), D_n(\theta)$ be the quasi-support functions for $K_n$ as defined in definition~\ref{defn:sptfunction}. Note by the classical Steinhaus theorem, we know that $\inf_{\theta \in S^{d-1}} r_1(\theta)=r > 0$ for all $\theta \in S^{d-1}$.
 
Then $0 < r_1(\theta) \leq r_2(\theta) \leq r_3(\theta) \leq \dots$ and $D_1(\theta) \leq D_2(\theta) \leq D_3(\theta) \leq \dots $, and $r_n(\theta) \leq D_n(\theta) \leq \frac{D}{2}$ for all $\theta \in S^{d-1}$.

Let the quasi-support functions for $Conv(K_1)$ be denoted $r_c(\theta)=D_c(\theta)$ and 
$$r_{\infty} = \lim_{n \to \infty} r_n, D_{\infty} = \lim_{n \to \infty} D_n.$$

Apriori we have $0 < r_{\infty}(\theta) \leq D_{\infty}(\theta) \leq D_c(\theta) \leq \frac{D}{2}$ for all $\theta \in S^{d-1}$.

\vskip.125in 

Now note for any $n \geq 1$ and any nonzero $u=|u|\theta(u) \in K_n$ (here $\theta(u)=\frac{u}{|u|} \in S^{d-1}$), we can restrict the Banach-Steinhaus process onto the part of $K_n$ on the line segment between $u$ and $-u$. This will be effectively a compact 1-dimensional space, symmetric about the origin and with diameter $2|u| \leq D$ which contains a line segment of diameter $r$ which is symmetric about the origin. Fix $\epsilon > 0$, Then Theorem~\ref{theorem:1DBSProcess} and its proof show that there is $m_0$ which only depends on $D, \epsilon, r$ such that $Star_{m_0+n}$ is $\epsilon$-close to all such $u$. Thus $Star_n \subseteq K_n \subseteq Star_{n+m_0}^{\epsilon} \subseteq K_{n+m_0}^{\epsilon}$ where $m_0$ only depends on $\epsilon$ (and implicitly $D,r$) but not on $n$.

 From here, it is easy to see that $\lim_{n \to \infty} Star_n = Conv(K_1)$ with respect to the Hausdorff metric. This is because for a given $\epsilon > 0$, we already know from Proposition~\ref{pro:Steinhausmap}, that 
$$K_n \subseteq Conv(K_1) \subseteq K_n^{\epsilon} \subseteq Star_{n+m_0}^{2\epsilon}$$ and 
$$Star_{n+m_0} \subseteq K_{n+m_0} \subseteq Conv(K_1)$$ for all large enough $n$. 
 
For the last part, take $x \in Conv(K_1)$ and choose $n$ such that 
$$n \geq max \left(log_2(\frac{2Dd}{r}), log_2(d(d+1))\right)$$ so that $\frac{r}{2} \geq \frac{Dd}{2^n}$. Then by Item 6 of Proposition~\ref{pro:Steinhausmap}, $x=y+e$ where $y \in K_n$ and $e \in \bar{B}(0,\frac{r}{2}) \subseteq K_1 \subseteq K_n$ i.e. $|e| \leq \frac{r}{2}$. Then $\frac{x}{2} = \frac{1}{2}(y+e) \in \frac{1}{2}(K_n + K_n)=K_{n+1}$. Thus we conclude $\frac{1}{2}Conv(K_1) \subseteq K_{n+1}$.

\end{proof}

\begin{corollary}
Let $K_0 \in Cpt_d$ of diameter $D$ and let $K_n=S^{[n]}(K_0)$ denote the Banach-Steinhaus process. Assume $K_0$ has {\bf positive Lebesgue measure}.  Let $f: Conv(K_1) \to \mathbb{R}$ be continuous then 
there exist $m \leq M$ such that $f(Conv(K_1))=[m,M]$. for any $\epsilon > 0$, for large enough $n$ we have:
$$
[m+\epsilon, M-\epsilon] \subseteq f(K_n) \subseteq [m,M].
$$
\end{corollary}
\begin{proof}
$Conv(K_1)$ is compact as it is the continuous image of $\Delta_d \times K_1^{d+1}$ by Caratheodory's theorem. Here $\Delta_d =\{ (t_0,\dots,t_d) | t_i \geq 0, \sum_{i=0}^d t_i=1 \}$ is the standard $d$-dimensional simplex which is compact. $Conv(K_1)$ is also clearly path connected, and so $f(Conv(K_1)) \subseteq \mathbb{R}^1$ is compact and connected and hence of the form $[m,M]$ by the intermediate and extreme value theorems.

Now as $Star_n$ is also compact and path connected for any $n$, and $Star_n \subseteq K_n \subseteq Conv(K_1)$, it follows that $f(Star_n)$ is a closed subinterval of $[m,M]$. As $f$ is uniformly continuous on $Conv(K_1)$ and since $Star_n$ converges to $Conv(K_1)$ in the Hausdorff metric, it follows that $f(Star_n)$ converges to $f(Conv(K_1))$ in the Hausdorff metric. From this and since $f(Star_n) \subseteq f(K_n)$, it follows that for any $\epsilon > 0$, we have 
$$
[m+\epsilon, M-\epsilon] \subseteq f(K_n) \subseteq [m,M]
$$
for all large enough $n$.
\end{proof}

Example~\ref{example:2points} shows that the last corollary does not hold when $K_0$ has measure $0$ even when $f(x)=x$ and we are in one dimension. It is crucial that $K_0$ have positive measure so that $K_1$ contains a ball of positive radius around the origin in order for the Banach-Steinhaus process to yield a sequence of sets "well-approximated" by their star-convex parts.

\section{Kakeya Needle Problem and Blaschke-Santal\'o diagrams}

\begin{definition}
Suppose $K$ is a convex body. For any $r > 0$, $K$ is called a $r$-Kakeya set, if it contains a closed interval of length $r$ in any given direction.
By convexity, this is equivalent to the condition that for any $u \in S^{d-1}$, there exists $k_1, k_2 \in K$ such that $k_1-k_2=ru$.

By convexity, this happens if and only if 
$$\{  su | 0 \leq s \leq r, u \in S^{d-1} \} \subset K-K$$ and hence is equivalent to $r(K) \geq r$.

\vskip.125in 

Thus for convex bodies $K$, $r(K)$, the radius of the largest closed ball centered about the origin contained in $K-K$ is also the largest $r$ for which $K$ is a $r$-Kakeya set.
\end{definition}

Now given a convex body $K$ which is a $r$-Kakeya set, $\frac{1}{r} K$ is a convex body which is a $1$-Kakeya set.

\begin{definition}
A convex body $K \subset \mathbb{R}^2$ is a $r$-Kakeya needle set in the plane, if there is a closed interval of radius $r$ in $K$, that can be continously rotated through 180 degrees within $K$.

Note that the condition of being a $r$-Kakeya needle set is apriori stronger than of being a $r$-Kakeya set.
\end{definition}

J. Pal showed in 1920, that the smallest area convex $1$-Kakeya needle set in the plane is an equilateral triangle of height $1$ and hence area $\frac{1}{\sqrt{3}}$. If the convexity assumption is removed, Besicovitch showed there are 
$1$-Kakeya sets of arbitrarily small measure and in fact even measure zero. Scaling Pal's result, we conclude that any $r$-Kakeya needle convex body $K$ in the plane must have area $|K|$ with  $|K| \geq \frac{r^2}{\sqrt{3}}$, yielding the inequality 
$$
r_{needle} \leq |K|^{0.5} 3^{0.25}
$$
for convex bodies $K$ in the plane where $r_{needle}$ is the largest $r$ for which $K$ is a $r$-Kakeya needle set.

\vskip.25in 

More generally, one may ask questions about the relationship of various quantities as one ranges over all nonempty, compact convex subsets $K$ of the plane. Typical quantities are like the area $A$, in-radius $r_{in}$ (largest radius of a closed disk contained in $K$), the circumradius $R$ (smallest radius of closed disk containing $K$), diameter $D$, Perimeter length $P$ of the boundary $\partial K$ (in case $K$ has nonempty interior) and width $W$ which is the smallest distance between two parallel lines such that $K$ lies between the lines.

The traditional Blaschke-Santal\'o diagram for planar convex bodies describes a region defined by 4 inequalities that constrain the three quantities $(r_{in}, R, D)$ by delineating exactly all the possibilities for $(\frac{r_{in}}{R}, \frac{D}{2R})$ as we range over all convex bodies in the plane, (see \cite{B1915}, \cite{BM2017}, \cite{F2023}).
This region is not convex but has 4 corners where two inequalities are sharp and equal corresponding to the cases where $K$ is a closed interval, a closed ball, an equilateral triangle or a Reuleaux triangle respectively.

Different authors have described similar "Blaschke-Santal\'o" diagrams for different triplets of quantities and these diagrams are often used in the area of shape-optimization. For example the triplets of (area, perimeter, moment intertia in \cite{HGL2023}) and (area, perimeter, Cheeger constant in \cite{F2023}).





\section{Proof of Theorem \ref{steinhausradiusthm} and Theorem \ref{mainmodified}} 

\vskip.125in 

Consider 
$$ F(x)=\int \chi_U(x+y) \chi_U(y) dy$$ and observe that 
$$ F(x)-|U|=\int \left( \chi_U(x+y)-\chi_U(y) \right) \chi_U(y) dy.$$ 

Also note that $F(x)$ is supported on $U-U$. We have

$$ {|F(x)-|U||}^2={\left( \int \left( \chi_U(x+y)-\chi_U(y) \right) \chi_U(y) dy \right)}^2$$
$$ \leq |U| \cdot \int {|\chi_U(x+y)-\chi_U(y)|}^2 dy$$
$$=|U| \cdot |\{(U-x) \backslash U \cup U \backslash (U-x)\}|=|U|M_U(x)$$ where the second line follows from the first by Cauchy-Schwartz and the third line is a direct calculation. We have 
$$ F(x)=|U|+(F(x)-|U|),$$ so $F(x)>0$ if 
$$ {|U|}^{\frac{1}{2}} \cdot M^{\frac{1}{2}}_U(x)<|U|,$$ which is the case if $M_U(x)<|U|$. It follows that 

$$ r(U) \ge \sup \left\{|x|: M_U(x)<|U| \right\}. $$

In particular, using that $M_U(x) \leq M^*(U)|x|$, we see that 
$$ r(U) \ge \frac{|U|}{M^{*}(U)}$$ and the proof is complete.

\end{document}